\documentclass[10pt,a4paper]{article}

\usepackage[vmargin=1.5in,hmargin=1in,centering]{geometry}

\usepackage{datetime}

\usepackage[all]{xy}

\usepackage{amssymb, amsmath, amsthm}

\usepackage{mathptmx}
\usepackage[scaled=.90]{helvet}
\usepackage{courier}

\numberwithin{equation}{section}

\usepackage[pdftex]{hyperref}

\hypersetup{
  colorlinks = true,
  linkcolor = black,
  urlcolor = blue,
  citecolor = black,
}

\newcommand{\blow}[2]{\overline{{#2}\setminus{#1}}}

\renewcommand{\emptyset}{\ensuremath{\text{\O}}}

\renewcommand{\ge}{\geqslant}
\renewcommand{\phi}{\varphi}

\newcommand{\bbC}{\mathbb{C}}
\newcommand{\bbH}{\mathbb{H}}

\newcommand{\bbR}{\mathbb{R}}
\newcommand{\bbZ}{\mathbb{Z}}

\newcommand{\rmF}{\mathrm{F}}
\newcommand{\rmH}{\mathrm{H}}

\newcommand{\rmQ}{\mathrm{Q}}
\newcommand{\rmS}{\mathrm{S}}
\newcommand{\rmT}{\mathrm{T}}

\DeclareMathOperator{\SO}{SO}

\DeclareMathOperator{\Map}{Map}
\DeclareMathOperator{\Imm}{Imm}
\DeclareMathOperator{\Emb}{Emb}
\DeclareMathOperator{\Mono}{Mono}
\DeclareMathOperator{\Blow}{Bl}
\DeclareMathOperator{\colim}{colim}

\newtheorem{theorem}{Theorem}[section]
\newtheorem{proposition}[theorem]{Proposition}
\newtheorem{corollary}[theorem]{Corollary}
\newtheorem{lemma}[theorem]{Lemma}
\newtheorem{theorem*}{Theorem}
\newtheorem{letterthm}{Theorem}

\theoremstyle{definition}

\newtheorem{example}[theorem]{Example}

\newtheorem{remark}[theorem]{Remark}
\newtheorem*{question*}{Question}

\newtheorem*{notation}{Notation}

\title{Spaces of knotted circles and exotic smooth structures}

\author{Gregory Arone \and Markus Szymik}

\newdateformat{mydate}{\monthname~\twodigit{\THEYEAR}}
\date{\mydate\today}

\begin{document}

\maketitle


\renewcommand{\abstractname}{\vspace{-2\baselineskip}}

\begin{abstract}\noindent%
Suppose that~$N_1$ and~$N_2$ are closed smooth manifolds of dimension~$n$ that are homeomorphic. We prove that the spaces of smooth knots~$\Emb(\rmS^1, N_1)$ and~$\Emb(\rmS^1, N_2)$ have the same homotopy~$(2n-7)$--type. In the~$4$--dimensional case this means that the spaces of smooth knots in homeomorphic~$4$--manifolds have sets~$\pi_0$ of components that are in bijection, and the corresponding path components have the same fundamental groups~$\pi_1$. The result about~$\pi_0$ is well-known and elementary, but the result about~$\pi_1$ appears to be new. The result gives a negative partial answer to a question of Oleg Viro. Our proof uses the Goodwillie--Weiss embedding tower. We give a new model for the quadratic stage of the Goodwillie--Weiss tower, and prove that the homotopy type of the quadratic approximation of the space of knots in~$N$ does not depend on the smooth structure on~$N$. Our results also give a lower bound on~$\pi_2\Emb(\rmS^1, N)$. We use our model to show that for every choice of basepoint, each of the homotopy groups~$\pi_1$ and~$\pi_2$ of~$\Emb(\rmS^1, \rmS^1\times\rmS^3)$ contains an infinitely generated free abelian group.





\end{abstract}


Oleg Viro asked: is the algebraic topology of the space of smooth~$1$--knots in a~$4$--manifold sensitive to the smooth structure on the ambient manifold~\cite{Viro}? More generally: can the homotopy type of the embedding space~$\Emb(\rmS^1,N)$ of knotted circles in a manifold~$N$ detect exotic smooth structures on~$N$? One of our main results answers these negatively in a range~(see Corollary~\ref{cor:main} below):

\begin{letterthm}\label{theorem: main}
Let~$N$ be a smooth manifold of dimension~$n$. The homotopy~\hbox{$(2n-7)$}--type of the space~$\Emb(\rmS^1,N)$ of smooth embeddings of the circle into~$N$ does not depend on the smooth structure.
\end{letterthm}

Recall that two spaces have the same homotopy~$k$--type if their~$k$--th Postnikov sections are homotopy equivalent; in particular, their homotopy groups~$\pi_*$ are isomorphic for~\hbox{$*\leqslant k$}. The theorem has content only for~$n\geqslant4$. In particular, in dimension~$n=4$, which is the context of Viro's original question, our result says that the spaces of knotted circles in two homeomorphic~$4$--manifolds have sets of components that are in bijection and that the corresponding components have isomorphic fundamental groups (see Corollary~\ref{cor:main4}). For example, this implies that, if~$\Sigma^4$ is any homotopy~$4$--sphere, then the space~$\Emb(\rmS^1,\Sigma^4)$ of knots in~$\Sigma^4$ is simply-connected (see Proposition~\ref{Prop:exotic-4-spheres}), regardless of the smooth structure.


Our proof of Theorem~\ref{theorem: main} is based on the manifold functor calculus developed by Goodwillie, Klein, and Weiss~(see~\cite{Weiss1996}, \cite {Weiss1999}, \cite{GoodwillieKleinWeiss03}, and~\cite{GoodwillieKlein15}). Let~$\rmT_2 \Emb(M,N)$ be the second (i.e., quadratic) approximation of the embedding tower constructed in~\cite{Weiss1999}. In Section~\ref{sec:quadratic}, we give a new description of the space~$\rmT_2\Emb(M,N)$ as a homotopy pullback~(see Theorem~\ref{thm:T2description}). 

\begin{letterthm}\label{letterthm: pullback}
Let~$M$ and~$N$ be closed smooth manifolds. There is a homotopy pullback square
\[
\xymatrix{
\rmT_2\Emb(M,N)\ar[r]\ar[d] & \Map_{\Sigma_2}(M^{[2]}, \ N^{[2]})\ar[d]\\
\Imm(M,N)\ar[r] &  \Map_{\Sigma_2}((M^{[2]}, \rmS(M)),(N\times N, \ N^{[2]})).
}
\]
\end{letterthm}

Here~$\Imm(M,N)$ denotes the space of immersions of~$M$ into~$N$, while~$M^{[2]}$ is the spherical blowup of~$M\times M$ at the diagonal, and~$\rmS(M)$ is the spherical tangent bundle of~$M$~(see Sections~\ref{sec:blowups} and~\ref{sec:quadratic} for more details).

There is a well-known description of the quadratic approximation~$\rmT_2\Emb(M,N)$ that goes back to Haefliger~\cite[Thm.~1.2.1]{GoodwillieKleinWeiss01}. The description in Theorem~\ref{letterthm: pullback} has a similar flavor, but is not identical to Haefliger's. Perhaps its main feature  is that it isolates the extent to which~$\rmT_2\Emb(M,N)$ depends on the tangential structure of~$M$ and~$N$. Also, Dax~\cite[VII.2.1]{Dax}, distilling Haefliger's double point elimination methods into a bordism theory, has given a refined description of the homotopy groups of the homotopy fiber of the inclusion of an embedding space into an immersion space in a range that is similar to ours. However, note that this does not solve our problem because it does not explain how this fiber is~`attached' to the immersion space. There are several other results~(see~\cite{Lashof} and~\cite{GoodwillieKleinWeiss03}, for example) suggesting that we can stratify the embedding space into pieces that are more or less obviously homeomorphism invariants. None of these arguments implies that also the extension problems can be solved, and it is in this direction where we leverage the geometry of the blow-up construction to address the first of them successfully. 

For embeddings of the circle~$M=\rmS^1$, we can use Theorem~\ref{letterthm: pullback} to show that the homotopy type of~$\rmT_2\Emb(S^1,N)$ is independent of the smooth structure on~$N$. The following theorem is Theorem~\ref{thm:main} in the text.

\begin{letterthm}\label{letterthm: invariance}
Let~$N_1$ and~$N_2$ be smooth~$n$--dimensional manifolds that are homeomorphic. Then the quadratic approximations~$\rmT_2\Emb(\rmS^1,N_1)$ and~$\rmT_2\Emb(\rmS^1,N_2)$ are homotopy equivalent.
\end{letterthm}

Theorem~\ref{theorem: main} is then a consequence of Theorem~\ref{letterthm: invariance} and known estimates of the connectivity of the approximation map~$\Emb(M,N)\to \rmT_2\Emb(M,N)$. In fact, Theorem~\ref{letterthm: invariance} implies a slightly stronger conclusion than what we stated in Theorem~\ref{theorem: main}. In particular, if~$N$ is~$4$--dimensional then the approximation map~\hbox{$\Emb(S^1,N)\to \rmT_2\Emb(S^1,N)$} is not just an isomorphism on~$\pi_0$ and~$\pi_1$ but is also an epimorphism on~$\pi_2$. Therefore one can use Theorem~\ref{letterthm: invariance} to derive a lower bound on~$\pi_2(\Emb(S^1,N))$. We calculate some examples in Section~\ref{sec:four}: we show that the space~$\Emb(\rmS^1,\Sigma^4)$ of knots in any homotopy~$4$--sphere~$\Sigma^4$ is simply-connected (see Proposition~\ref{Prop:exotic-4-spheres}) and we give an example of a~$4$--manifold~$N$ such that the map~$\Emb(\rmS^1,N)\to \Imm(\rmS^1,N)$ has infinitely generated kernels on~$\pi_1$ and~$\pi_2$ (see Corollary~\ref{cor:kernels}).


The outline of this paper is as follows. In Section~\ref{sec:blowups}, we review some results on tangent bundles and blow-ups that we will use in the later sections. Section~\ref{sec:immersions} contains a discussion of spaces of immersions; this both illustrates our general strategy and provides results that we use later on. Section~\ref{sec:quadratic} is the center of the text, where we give a new description of the quadratic approximation~$\rmT_2\Emb(M,N)$ that is valid for all~$M$ and~$N$ of any dimension. We return in Section~\ref{sec:circle} to the case where the source~$M=\rmS^1$ is the circle, to deduce our main results for spaces of knotted circles in general targets. The final Section~\ref{sec:four} specializes further to the case where the target~$N$ is a smooth~$4$--manifold, to give more specific examples in Viro's original context. In particular, we show that for every choice of basepoint, each of the homotopy groups~$\pi_1$ and~$\pi_2$ of~$\Emb(\rmS^1, \rmS^1\times\rmS^3)$ contains an infinitely generated free abelian group.


\section{Blow-ups}\label{sec:blowups}

The results in this section are valid for manifolds of all dimensions. Only in the last subsection shall we work out the example of the circle in sufficient detail for later use.

\begin{notation}
Let~$A$ be a submanifold of a manifold~$X$. We will denote the spherical normal bundle of~$A$ by~$\rmS_A(X)$, and the spherical blowup of~$X$ at~$A$ by~$\blow{A}{X}$. 
\end{notation}
Recall that~$\blow{A}{X}$ is a manifold with boundary whose interior is the complement~$X\setminus A$ and whose boundary is~$\rmS_A(X)$. There is a commutative diagram 
\[
\xymatrix{
X\setminus A\ar@{=}[d]\ar[r]^-\sim & \blow{A}{X}\ar[d] & \rmS_A(X)\ar[d]\ar[l]_-\supset\\
X\setminus A\ar[r]_-\subset& X & A\ar[l]^-\supset
}
\]
The spherical blowup is locally modeled as follows: take an inclusion~$U\subset V$ of linear spaces. Then there is an inclusion~$j\colon V\setminus U\to V$, and a projection~\hbox{$q\colon V\setminus U\to\rmS_U(V)$}. The blowup~$\blow{U}{V}$ is the closure of the image of the natural map
\[
(j,q)\colon V\setminus U\longrightarrow V\times\rmS_U(V).
\]

\begin{proposition}\label{prop:functoriality}
If~$B$ is a submanifold of another manifold~$Y$, and if~\hbox{$f\colon X\to Y$} is a smooth map with~$f^{-1}B=A$ that induces~(via the derivative) a fiberwise monomorphism between normal bundles, then it induces a smooth map~$\blow{A}{X}\to\blow{B}{Y}$.
\end{proposition}

The induced map~$\blow{A}{X}\to\blow{B}{Y}$ is defined by using the restriction of the map~$f$ between the interiors~\hbox{$X\setminus A$} and~$Y\setminus B$, and the map induced by the derivative~$f'$ between the spherical normal bundles on the boundaries. More details can be found for example in~\cite{AroneKankaanrinta}.


We will mostly be interested in the case when~$X=N\times N$ is the product of a manifold~$N$ with itself, and~$A$ is the diagonal. 
\begin{notation}
We will denote the blow-up~$\blow{N}{N\times N}$ by~$N^{[2]}$.
\end{notation}
Note that if~$N$ is a closed manifold, then the boundary of ~$N^{[2]}$ is~$\rmS_N(N\times N)$, the spherical normal bundle of the diagonal in~$N\times N$, which can be identified with the spherical tangent bundle of~$N$. 
\begin{notation}
We will denote the spherical tangent bundle of~$N$ by~$\rmS(\tau N)$ or just~$\rmS(N)$.
\end{notation}
Thus there is a canonical homeomorphism~$\rmS(N)\cong \rmS_N(N\times N)$. We identify~$\rmS(N)$ with the boundary of~$N^{[2]}$. Note that the pair~$(N^{[2]}, \rmS(N))$ has a canonical action by the group~$\bbZ_2$. 

Locally, we have the following situation:
\begin{example}\label{ex:Euclidean}
In the case when~$M=\bbR^m$ is the local model, a linear transformation gives
\[
\bbR^m\times\bbR^m\setminus\Delta
\cong\bbR^{2m}\setminus\bbR^m
\cong\bbR^m\times(\bbR^m\setminus0)
\cong\bbR^m\times\rmS^{m-1}\times]\,0,\infty\,[.
\]
The involution is free: it is the antipodal action on~$\rmS^{m-1}$ and trivial on all other factors.
In this model, we have
\[
\rmS(\bbR^m)\cong\bbR^m\times\rmS^{m-1}\times\{0\}
\]
and
\[
(\bbR^m)^{[2]}\cong\bbR^m\times\rmS^{m-1}\times[\,0,\infty\,[,
\]
so that the boundary inclusion of~$\rmS(\bbR^m)$ into~$(\bbR^m)^{[2]}$ is a~$\Sigma_2$--homotopy equivalence. In fact, both of these spaces are~$\Sigma_2$--homotopy equivalent to~$\rmS^{m-1}$ with the antipodal action.
\end{example}


The following simple proposition is one of the main technical results of this paper.

\begin{proposition}\label{prop:homotopy_boundary_new}
If~$M$ and~$N$ are smooth closed manifolds that are homeomorphic to each other, then the diagrams of spaces
\[
\rmS(M)\to M^{[2]}\to M\times M
\] 
and
\[
\rmS(N)\to N^{[2]}\to N\times N
\] 
are connected by a zig-zag of~$\Sigma_2$--equivariant homotopy equivalences.
\end{proposition}

\begin{proof}
We shall use several times that every open~$\Sigma_2$--neighborhood of the diagonal contains a tubular~$\Sigma_2$--neighborhood.

To start with, we choose any tubular~$\Sigma_2$--neighborhood~$A$ such that
\[
M\subseteq A\subseteq M\times M.
\]
Here and elsewhere we identify~$M$ with the diagonal of~$M\times M$. Let~$f\colon M\to N$ be a homeomorphism, and let
\[
h=f\times f\colon M\times M\longrightarrow N\times N
\]
be its square. This is~$\Sigma_2$--equivariant, so that we get an open~$\Sigma_2$--neighborhood~$h(A)$ of the diagonal within~$N\times N$. We choose another tubular~$\Sigma_2$--neighborhood~$B$ such that
\[
N\subseteq B\subseteq h(A)\subseteq N\times N.
\]
We repeat this process twice more and find tubular~$\Sigma_2$--neighborhoods~$C$ and~$D$ of the respective diagonals and end up with a chain
\[
M\subseteq h^{-1}(D)\subseteq C\subseteq h^{-1}(B)\subseteq A\subseteq M\times M.
\]
Once this is set up, we consider the three inclusions
\begin{equation}\label{eq:three}
h^{-1}(D)\setminus M\longrightarrow
C\setminus M\longrightarrow
h^{-1}(B)\setminus M\longrightarrow
A\setminus M.
\end{equation}
We have~$h^{-1}(D)\setminus M\cong D\setminus N\simeq \rmS(N)$ and similarly~$h^{-1}(B)\setminus M\simeq \rmS(M)$. The  composition of the first two inclusions in~\eqref{eq:three}, is an equivalence. Similarly, the composition of the last two maps in~\eqref{eq:three} is an equivalence (this time of spaces equivalent to~$\rmS(M)$). It follows that the inclusion~\hbox{$C\setminus M\to h^{-1}(B)\setminus M$} in the middle is also an equivalence of subsets of~\hbox{$M\times M$}. Thus~$h$ induces an equivalence of diagrams
\[
\xymatrix{
C\setminus M \ar[d]^\simeq\ar[r] & M\times M\setminus M \ar[d]^\simeq\ar[r] & M\times M\ar[d]^\simeq\\
B\setminus N\ar[r] & N\times N\setminus N \ar[r]& N\times N
}
\]
Next, observe that there is also a zig-zag of~$\Sigma_2$--equivariant equivalences of diagrams
\[
\xymatrix{
C\setminus M \ar[d]^\simeq\ar[r] & M\times M\setminus M \ar[d]^\simeq\ar[r] & M\times M\ar[d]^\simeq\\
\blow{M}{C}\ar[r] & M^{[2]} \ar[r]& M\times M\\
\rmS(M)\ar[r]\ar[u]_\simeq & M^{[2]} \ar[r]\ar[u]_\simeq& M\times M\ar[u]_\simeq
}
\]
and similarly there is a zig-zag of equivalences connecting the diagrams 
\[
B\setminus N\longrightarrow  N\times N\setminus N\longrightarrow N\times N
\] 
and
\[
\rmS(N)\longrightarrow N^{[2]}\longrightarrow N\times N.
\]
\end{proof}

\begin{remark}
It seems likely that the assumption that~$M$ and~$N$ are closed manifolds can be relaxed.
\end{remark}


\begin{example}\label{ex:circle}
We need to understand the pair~$(N^{[2]}, \ \rmS(N))$ in the case~$N=\rmS^1$. The complement of the diagonal in the torus\hbox{~$\rmS^1\times\rmS^1$} consists of the ordered pairs of distinct points on the circle, and this is homeomorphic to~$\rmS^1\times]\,0,2\pi\,[$ under the map that sends a pair of distinct points to the pair consisting of the first point and the angle to the second point (counter-clock-wise, say). The involution that interchanges the two points is given, in this model, by~$(z,t)\mapsto(z\exp(t\mathrm{i}),2\pi-t)$, and it obviously extends to the spherical blowup, which is the cylinder~$\rmS^1\times[\,0,2\pi\,]$. Note that the involution interchanges the two boundary components via~$(z,t)\mapsto(z,2\pi-t)$ for~$t\in\{0,2\pi\}$ and acts on the central circle as~$(z,\pi)\mapsto(-z,\pi)$. To summarize, there is a homeomorphism
\[
((\rmS^1)^{[2]}, \ \rmS(\rmS^1))\cong (\rmS^1\times  [\,0,2\pi\,], \rmS^1\times \{0, 2\pi\}).
\]
Notice also that the blow-up~$(\rmS^1)^{[2]}$ is~$\Sigma_2$--equivariantly homotopy equivalent to the circle~$\rmS^1$ with the antipodal involution.
\end{example}


\section{Linear approximation: immersions}\label{sec:immersions}

In this section we point out that while in general~$\Imm(M,N)$ is sensitive to the smooth structure on~$N$, the space~$\Imm(\rmS^1,N)$ is not. This is true for target manifolds~$N$ of all dimensions. 

Let~$M$ and~$N$ be smooth manifolds. Let~$\Mono(\tau M, \tau N)$ denote the space of monomorphisms from the tangent bundle of~$M$ into the tangent bundle of~$N$. Differentiation induces a natural map~\hbox{$\Imm(M, N) \to \Mono(\tau M, \tau N)$}. It is well-known from Hirsch--Smale theory that this map is an equivalence if~$\dim(N)>\dim(M)$. One can identify~$\Mono(\tau M, \tau N)$ with~$\rmT_1\Emb(M, N)$, the first stage in the Goodwillie--Weiss tower of approximations of~$\Emb(M, N)$~\cite{Weiss1999}.

In the case~$M=\rmS^1$ we obtain that there are equivalences
\[
\Imm(\rmS^1, N)\simeq \Mono(\tau \rmS^1, \tau N)\simeq \Lambda \rmS(N),
\]
where~$\Lambda$ denotes the free loop space functor and~$\rmS(N)$ is, as usual, the sphere tangent bundle of~$N$.

The tangent bundle of a smooth manifold~$N$ is {\em not} a topological invariant: Milnor~\cite[Cor.~1]{Milnor:ICM} showed that there are smooth manifolds that are homeomorphic, but where one of them is parallelizable, and the other one is not. In other words, there are smooth structures on some topological manifold that afford non-isomorphic tangent bundles.

On the other hand, the {\em sphere} bundle is to some extent a topological invariant. The following result is a corollary of theorems of  Thom~\cite[Cor.~IV.2]{Thom} and Nash~\cite{Nash}. It also follows from our Proposition~\ref{prop:homotopy_boundary_new}.

\begin{proposition}\label{prop:sphere_bundles}
If smooth manifolds~$M$ and~$N$ are homeomorphic, then the total spaces of the spherical tangent bundles~$\rmS(M)$ and~$\rmS(N)$ are homotopy equivalent.
\end{proposition}

\begin{theorem}\label{thm:circle_immersions}
The homotopy type of the space~$\Imm(\rmS^1,N)$ of immersion of the circle into a smooth manifold~$N$ does not depend on the smooth structure of~$N$.
\end{theorem}

\begin{proof}
We saw that the space~$\Imm(\rmS^1,N)$ is homotopy equivalent to 
$\Lambda \rmS(N)$. Now the result follows from Proposition~\ref{prop:sphere_bundles}.
\end{proof}

Goodwillie and Klein~\cite{GoodwillieKlein15} have shown that the connectivity of the map
\[
\Emb(M,N)\longrightarrow\rmT_k\Emb(M,N)
\]
to the~$k$--th layer in the Goodwillie--Weiss tower is at least~$k(n-m-2)+1-m$. Recall: a map is called~{\it$c$--connected} if all of its homotopy fibers are~$(c-1)$--connected. In particular, the Goodwillie--Klein result implies the much more elementary fact that the inclusion~\hbox{$\Emb(M,N)\to\Imm(M,N)$} is~\hbox{$(n-2m-1)$}--connected. It follows for~\hbox{$M=\rmS^1$} that the inclusion~\hbox{$\Emb(\rmS^1,N)\to\Imm(\rmS^1,N)$} is~$(n-3)$--connected. Theorem~\ref{thm:circle_immersions} implies that the homotopy~$(n-4)$--type of~$\Emb(\rmS^1,N)$ does not depend on the smooth structure of~$N$. In the following sections, we shall roughly double this range.


We end this section with a couple of elementary observations about the set~$\pi_0$ of components and the fundamental groups~$\pi_1$ of~$\Imm(\rmS^1, N)$. Let~$n$ be the dimension of~$N$. The bundle map~$\rmS(N)\to N$ is~$(n-1)$--connected. It follows that there is an~$(n-2)$--connected map 
\[
\Imm(\rmS^1, N)\simeq \Lambda \rmS(N)\to \Lambda(N).
\]
Assuming~$n\ge 4$, we have isomorphisms~$\pi_i\Imm(\rmS^1, N)\cong\pi_i\Lambda(N)$ for~$i=0,1$. Using well-known facts about the homotopy groups of~$\pi_i\Lambda(N)$, we obtain the following proposition.

\begin{proposition}\label{prop:immersion_homotopy_groups}
Let~$N$ be a connected smooth manifold of dimension~$n\geqslant4$. Then the set of components of the space~$\Imm(\rmS^1, N)$ is in natural bijection with the set of conjugacy classes of elements in the fundamental group of~$N$. If~$N$ is simply-connected, then the fundamental group of the space of immersions is isomorphic to~$\pi_2N\cong\rmH_2(N;\bbZ)$.
\end{proposition}


\section{Quadratic approximations}\label{sec:quadratic}

In this section, we will give a new description of the quadratic approximation~$\rmT_2\Emb(M,N)$ that is valid for all~$M$ and~$N$.

The first order (a.k.a.~linear) approximation to the space~$\Emb(M,N)$ of embeddings~$M\to N$ is given by the space~$\Mono(\tau M, \tau N)$ of monomorphisms of tangent bundles~\cite{Weiss1999}, and the corresponding approximation map~\hbox{$\Emb(M, N)\to \Mono(\tau M, \tau N)$} is induced by differentiation. To understand the quadratic approximation we need to study the behavior of maps on pairs of points. 

Recall that we have two canonical~$\Sigma_2$--equivariant maps~$\rmS(M)\to M^{[2]}$ and~$M^{[2]}\to M\times M$. We define the space~~$\Map_{\Sigma_2}((M^{[2]}, \rmS(M)),(N\times N, \ N^{[2]}))$ as the space of commutative diagrams
\[
\xymatrix{
\rmS(M)\ar[r]^-\alpha\ar[d] & N^{[2]}\ar[d]\\
M^{[2]} \ar[r]_-\omega & N\times N
}
\]
of~$\Sigma_2$--equivariant maps, where the vertical arrows are canonical. In other words, there is a pullback square
\begin{equation}\label{equation: pullback}
\xymatrix{
\Map_{\Sigma_2}((M^{[2]}, \rmS(M)),(N\times N, \ N^{[2]}))\ar[r]\ar[d] & \Map_{\Sigma_2}(M^{[2]},N\times N)\ar[d]\\
\Map_{\Sigma_2}(\rmS(M),N^{[2]}) \ar[r] & \Map_{\Sigma_2}(\rmS(M),N\times N)
}
\end{equation}
Notice that the boundary inclusion~$\rmS(M)\to M^{[2]}$ is a~$\Sigma_2$--cofibration: it is a cofibration and the group~$\Sigma_2$ acts freely on the complement of the image. It follows that the right vertical map in~\eqref{equation: pullback} is a fibration, and that the square diagram~\eqref{equation: pullback} it is both a strict and a homotopy pullback.

\parbox{\linewidth}{\begin{lemma}
The homotopy type of
\[
\Map_{\Sigma_2}((M^{[2]}, \rmS(M)),(N\times N, \ N^{[2]}))
\]
for a fixed source~$M$, only depends on the homeomorphism type of~$N$.
\end{lemma}}

\begin{proof}
Since the inclusion~$N\times N\setminus N \rightarrow N^{[2]}$ is an equivalence, the space
$$\Map_{\Sigma_2}((M^{[2]}, \rmS(M)),(N\times N, \ N^{[2]}))$$ 
is homotopy equivalent to the homotopy pullback of the diagram
\[
\xymatrix{
 & \Map_{\Sigma_2}(M^{[2]},N\times N)\ar[d]\\
\Map_{\Sigma_2}(\rmS(M),N\times N\setminus N) \ar[r] & \Map_{\Sigma_2}(\rmS(M),N\times N)
}
\]
Clearly this homotopy pullback only depends on the homeomorphism type of~$N$.
\end{proof}

There is an evident commutative diagram
\[
\xymatrix{
\Map_{\Sigma_2}(M^{[2]},N^{[2]})\ar[r]\ar[d] & \Map_{\Sigma_2}(M^{[2]},N\times N)\ar[d]\\
\Map_{\Sigma_2}(\rmS(M),N\times N\setminus N) \ar[r] & \Map_{\Sigma_2}(\rmS(M),N\times N)
}
\]
This diagram induces a natural map.
\begin{equation}\label{eq:blow-to-P}
\Map_{\Sigma_2}(M^{[2]},N^{[2]})\longrightarrow\Map_{\Sigma_2}((M^{[2]}, \rmS(M)),(N\times N, \ N^{[2]})).
\end{equation}


There also exists a natural map.
\begin{equation}\label{eq:emb-to-blow}
\Emb(M,N)\longrightarrow\Map_{\Sigma_2}(M^{[2]},N^{[2]}).
\end{equation}
Indeed, for an embedding~$f\colon M\to N$, the map~\hbox{$f\times f\colon M\times M\to N\times N$} is automatically a fiberwise monomorphism on the normal bundle of the diagonal. Furthermore~$f\times f$ satisfies~$(f\times f)^{-1}{N}=M$. Therefore, we can use Proposition~\ref{prop:functoriality} to produce a map: the blow-up of~$f\times f$ at the diagonal.

Next, we claim that there is a commutative diagram
\begin{equation}\label{eq:Imm diagram}
\xymatrix{
\Mono(\tau M, \tau N)\ar[r]\ar[d] & \Map_{\Sigma_2}(M^{[2]},N\times N)\ar[d]\\
\Map_{\Sigma_2}(\rmS(M), N^{[2]}) \ar[r] & \Map_{\Sigma_2}(\rmS(M),N\times N)
}
\end{equation}
To define this diagram we have to specify the top horizontal and the left vertical maps. The top horizontal map is the composition of the following maps, each one of which is obvious
\[
\Mono(\tau M, \tau N)\to \Map(M, N) \xrightarrow{f\mapsto f\times f} \Map_{\Sigma_2}(M\times M, N\times N) \to \Map_{\Sigma_2}(M^{[2]},N\times N).
\]
The left vertical map is the following composition of obvious maps
\[
\Mono(\tau M, \tau N)\to \Map_{\Sigma_2}(\rmS(M), \rmS(N)) \to  \Map_{\Sigma_2}(\rmS(M), N^{[2]}).
\]
It is an easy exercise to check that with these definitions the diagram~\eqref{eq:Imm diagram} commutes. This diagram gives rise to a natural map
\begin{equation}\label{eq:imm-to-pullback}
\Mono(\tau M,\tau N) \longrightarrow \Map_{\Sigma_2}((M^{[2]}, \rmS(M)),(N\times N, \ N^{[2]}))
\end{equation}

\begin{lemma}\label{lem:square}
The following diagram commutes
\[
\xymatrix@C=40pt{
\Emb(M,N)\ar[r]^-{\eqref{eq:emb-to-blow}}\ar[d] & \Map_{\Sigma_2}(M^{[2]},N^{[2]})\ar[d]^-{\eqref{eq:blow-to-P}}\\
\Mono(\tau M, \tau N)\ar[r]^-{\eqref{eq:imm-to-pullback}} & \Map_{\Sigma_2}((M^{[2]}, \rmS(M)),(N\times N, \ N^{[2]})).
}
\]
\end{lemma}

\begin{proof}
Let~$f\colon M\rightarrow N$ be an embedding. Our task boils down to the question whether the diagram
\[
\xymatrix{
\rmS(M)\ar[d]\ar[r]^-{f'} & \rmS(N)\ar[r] & N^{[2]}\ar[d]\\
M^{[2]}\ar[r]\ar[rru]|-{f^{[2]}} & M\times M\ar[r]^{f\times f} & N\times N
}
\]
is commutative for all embeddings~$f\colon M\to N$. Here~$f^{[2]}$ is the blowup of~$f\times f$ at the diagonal, and the unlabelled arrows are the obvious ones. The diagram is commutative by definition of the map~$f^{[2]}$.
\end{proof}


The commutative square in Lemma~\ref{lem:square} is not a (homotopy) pullback in general. But it is in some important cases:

\begin{lemma}\label{lem:2-equivalence}
The commutative square in Lemma~\ref{lem:square} is a homotopy pullback if~$M$ is the disjoint union of at most two copies of~$\bbR^m$.
\end{lemma}

\begin{proof}
Let~$M\cong\underline{k}\times \bbR^m$, where~$\underline{k}$ is a finite set with~$k$ elements, and analyze the commutative square
\[
\xymatrix{
\Emb(M,N)\ar[r]\ar[d] &\Map_{\Sigma_2}(M^{[2]},N^{[2]})\ar[d]\\
\Mono(\tau M, \tau N)\ar[r] & \Map_{\Sigma_2}((M^{[2]}, \rmS(M)),(N\times N, \ N^{[2]})).
}
\]
in this case. Our goal is to show that when~$k\leqslant 2$ the square is a homotopy pullback. 

In the case the set~$\underline{k}=\underline{0}$ is empty, all the spaces involved are contractible and there is nothing to prove. 

Next, suppose that~$\underline{k}=\underline{1}$ is a singleton, so that we have a homeomorphism~$M\cong \bbR^m$. In this case the map~\hbox{$\Emb(\bbR^m ,N)\to \Mono(\tau\bbR^m ,\tau N)$} is an equivalence, so we need to show that the map 
\begin{equation}\label{eq:k=1}
\Map_{\Sigma_2}((\bbR^m)^{[2]},N^{[2]})\longrightarrow\Map_{\Sigma_2}(((\bbR^m)^{[2]}, \rmS(\bbR^m)),(N\times N, \ N^{[2]}))
\end{equation}
is a homotopy equivalence. This is equivalent to showing that the diagram 
\[
\xymatrix{
\Map_{\Sigma_2}((\bbR^m)^{[2]},N^{[2]})\ar[r]\ar[d] & \Map_{\Sigma_2}(\rmS(\bbR^m),N^{[2]})\ar[d]\\
\Map_{\Sigma_2}((\bbR^m)^{[2]},N\times N)\ar[r] & \Map_{\Sigma_2}(\rmS(\bbR^m),N\times N)
}
\]
is a homotopy pullback square. The boundary inclusion~$\rmS(\bbR^m)\rightarrow(\bbR^m)^{[2]}$ is a~$\Sigma_2$--equivariant homotopy equivalence~(see Example~\ref{ex:Euclidean}). Therefore, both horizontal arrows are homotopy equivalences, and then the square must be a homotopy pullback.

Lastly, let us suppose that the set~$\underline{k}=\underline{2}$ consists of two points. Let~$M$ denote~$\underline{k}\times \bbR^m$. There is an embedding of~$\underline{k}$ into~$M$, sending each point of~$\underline{k}$ to the origin of the corresponding copy of~$\bbR^m$. This embedding gives rise to the following diagram,
\[
\xymatrix@!0@C=80pt@R=30pt{
& \Emb(\underline{k} ,N)\ar[rr]\ar[dd] |\hole & & \Map_{\Sigma_2}(\underline{k}^{[2]},N^{[2]})\ar[dd] \\
\Emb(M,N)\ar[rr]\ar[ru]\ar[dd] & & \Map_{\Sigma_2}(M^{[2]},N^{[2]})\ar[dd]\ar[ru]\\
&N^k\ar[rr]|!{[ru];[rd]}\hole & &\Blow_\Delta(\underline{k} ,N)\\
\Mono(\tau M, \tau N)\ar[rr]\ar[ru] & &\Blow_\Delta(M, N)\ar[ru].
}
\]
where we have have used the abbreviation
\[
\Blow_\Delta(M ,N)=\Map_{\Sigma_2}((M^{[2]}, \rmS(M)),(N\times N, \ N^{[2]})).
\]
We want to prove that the front face is a homotopy pullback. For this it is enough to prove that the left, back, and right faces are each a homotopy pullback. 

It is a standard fact that the left face is a homotopy pullback for~$M=\underline{k}\times \bbR^m$, for any~$k$: the difference between an embedding of~$\bbR^m$ and an embedding of its center is given by a framing of the tangent space at the center, and it is the same for mere immersions.

Note that since~$\underline{k}$ is a zero dimensional manifold,~$\rmS_\Delta(\underline{k})=\emptyset$ and~$\underline{k}^{[2]}=\underline{k}\times \underline{k}\setminus\underline{k}$. In particular, for~$k=2$, we have~$\underline{k}^{[2]}\cong\Sigma_2$ as~$\Sigma_2$--sets. Then, by inspection, the back face is homeomorphic to the square
\[
\xymatrix{
N\times N\setminus N\ar[r]^-\subset\ar[d] & N^{[2]}\ar[d]\\
N\times N\ar[r]_-= & N\times N.
}
\]
Clearly the horizontal arrows are equivalences, and so we have a homotopy pullback square. 


As for the right face, the general formula
\[
(L\sqcup M)^{[2]}\cong (L\times M)\sqcup(M\times L)\sqcup L^{[2]}\sqcup M^{[2]}.
\]
gives that in the case~$M=\underline{2}\times \bbR^m$ there is a homeomorphism
\[
M^{[2]}\cong\left(\Sigma_2\times\bbR^{2m}\right)\sqcup\ \underline{2}\times\left((\bbR^m)^{[2]}\right),
\]
and then
\[
\Map_{\Sigma_2}(M^{[2]},N^{[2]})
\simeq
N^{[2]}\times\Map_{\Sigma_2}((\bbR^m)^{[2]},N^{[2]})^2.
\]
This maps to
\[
\Blow_\Delta(\underline{2}\times\bbR^m,N)
\simeq
(N\times N)\times\Blow_\Delta(\bbR^m,N)^2.
\]
Also, we have~$\rmS(M\sqcup N)\cong \rmS(M)\sqcup \rmS(N)$, so that
\[
\rmS(\underline{2}\times\bbR^m)\cong\underline{2}\times \rmS(\bbR^m), 
\]
and the boundary inclusion of~$\rmS(\underline{2}\times\bbR^m)$ into~$(\underline{2}\times\bbR^m)^{[2]}$ becomes the evident map
\[
\underline{2}\times \rmS(\bbR^m)\longrightarrow\left(\Sigma_2\times\bbR^{2m}\right)\sqcup\left(\underline{2}\times(\bbR^m)^{[2]}\right)
\]
into the summand on the right:~$\underline{2}\times$boundary inclusion of~$\rmS(\bbR^m)\rightarrow(\bbR^m)^{[2]}$.

Together, we see that the right face becomes
\[
\xymatrix{
N^{[2]}\times\Map_{\Sigma_2}((\bbR^m)^{[2]}, N^{[2]})^2\ar[r]\ar[d] & N^{[2]}\ar[d]\\
(N\times N)\times\Blow_\Delta(\bbR^m,N)^2\ar[r]& N\times N,
}
\]
We proved that the map 
\[
\Map_{\Sigma_2}((\bbR^m)^{[2]}, N^{[2]})\to \Blow_\Delta(\bbR^m,N)
\]
is a homotopy equivalence when we considered the case~$k=1$ (see~\eqref{eq:k=1}). It follows that the last square is a homotopy pullback.
\end{proof}


For general~$M$, we have the following result:

\begin{lemma}\label{lem:quadratic}
Define~$\rmF_N(M)$ to be the homotopy pullback of the following diagram
\[
\xymatrix{
 & \Map_{\Sigma_2}(M^{[2]},\ N^{[2]})\ar[d]^-{\eqref{eq:blow-to-P}}\\
\Mono(\tau M, \tau N)\ar[r] & \Map_{\Sigma_2}((M^{[2]}, \rmS(M)),(N\times N, \ N^{[2]})).
}
\]
The functor~$M\mapsto\rmF_N(M)$ is quadratic
\end{lemma}

\begin{proof}
The class of functors of degree at most~$d$ is closed under homotopy limits. Therefore, it is enough to prove that the three functors in the homotopy pullback defining~$\rmF_N$ are quadratic. 

The functor~$M\mapsto \Mono(\tau M, \tau N)$ is in fact linear. The functor
\[
M\longmapsto \Map_{\Sigma_2}(M^{[2]},N^{[2]})
\]
is quadratic essentially by~\cite[Ex.~2.4/7.1]{Weiss1999}. Finally, the functor
\[
M\longmapsto \Map_{\Sigma_2}((M^{[2]}, \rmS(M)),(N\times N, \ N^{[2]}))
\]
is, by definition~\eqref{equation: pullback}, itself a homotopy pullback of functors each one of which is easily shown to be of degree at most~$2$. 
\end{proof}

\begin{theorem}\label{thm:T2description}
For all~$M$ and~$N$, the commutative square in Lemma~\ref{lem:square} induces a homotopy pullback square
\[
\xymatrix{
\rmT_2\Emb(M,N)\ar[r]\ar[d] & \Map_{\Sigma_2}(M^{[2]},\ N^{[2]})\ar[d]\\
\Mono(\tau M, \tau N)\ar[r] &  \Map_{\Sigma_2}((M^{[2]}, \rmS(M)),(N\times N, \ N^{[2]})).
}
\]
\end{theorem}
\begin{proof}
Lemma~\ref{lem:2-equivalence} shows that the canonical map~$\Emb(M,N)\to\rmF_N(M)$ of functors in~$M$ induces an equivalence~$\rmT_2\Emb(M,N)\to\rmT_2\rmF_N(M)$ between their second order approximations. Lemma~\ref{lem:quadratic} shows that the canonical map~$\rmF_N\to\rmT_2\rmF_N$ is an equivalence. Both together imply the result.
\end{proof}


\section{Embeddings of the circle}\label{sec:circle}

For a general source manifold~$M$, we see no reason why the homotopy type of the quadratic approximation~$\rmT_2\Emb(M,N)$ should be independent of the smooth structure on~$N$. In this section, we will specialize to the case~$M=\rmS^1$, so that the embedding spaces are spaces of knots. The target manifold~$N$ can still be arbitrary of dimension at least~$4$.

We are ready to state and prove our main theorem.

\begin{theorem}\label{thm:main}
The homotopy type of the space~$\rmT_2\Emb(\rmS^1,N)$ does not depend on the smooth structure of the manifold~$N$.
\end{theorem}

\begin{proof} 
By Theorem~\ref{thm:T2description}, and together with the equivalence~$\Imm(S^1, N)\simeq \Lambda \rmS(N)$, we already know that the space~$\rmT_2\Emb(\rmS^1,N)$ is equivalent to the homotopy pullback of the following diagram.
\begin{equation} \label{eq: T2knots}
\xymatrix{
 & \Map_{\Sigma_2}((\rmS^1)^{[2]},N^{[2]})\ar[d]\\
\Lambda \rmS(N)\ar[r] & \Map_{\Sigma_2}(((\rmS^1)^{[2]}, \rmS(\rmS^1)),(N\times N, \ N^{[2]}))
}
\end{equation}
We claim that the homotopy type of this diagram~\eqref{eq: T2knots} is determined by the homotopy type of the diagram~\hbox{$\rmS(N)\to N^{[2]} \to N\times N$}. By Proposition~\ref{prop:homotopy_boundary_new}, the latter is determined by the homeomorphism type of~$N$ and is independent of the smooth structure.

Because the lower right corner of the diagram~\eqref{eq: T2knots} is defined as a pullback, that entire diagram is determined by the following diagram.
\[
\xymatrix@!0@C=9em@R=5em{
 & \Map_{\Sigma_2}((\rmS^1)^{[2]}, \ {N\times N})\ar[d]\\
\Lambda \rmS(N)\ar[ru]\ar[rd]& \Map_{\Sigma_2}( \rmS(\rmS^1), \ {N\times N}) & \Map_{\Sigma_2}((\rmS^1)^{[2]},N^{[2]}) \ar[lu]\ar[ld] \\ &  \Map_{\Sigma_2}(\rmS(\rmS^1),N^{[2]}) \ar[u]&
}
\]
Using the fact that~$\rmS(\rmS^1)\cong \Sigma_2 \times \rmS^1$, we may rewrite this diagram as follows
\begin{equation}\label{diagram: T2knots}
\xymatrix@!0@C=9em@R=5em{
 & \Map_{\Sigma_2}((\rmS^1)^{[2]}, \ {N\times N})\ar[d]\\
\Lambda \rmS(N)\ar[ru]^-f\ar[rd]_-g & \Map(\rmS^1, \ {N\times N}) & \Map_{\Sigma_2}((\rmS^1)^{[2]},N^{[2]}) \ar[lu]\ar[ld] \\ &  \Map(\rmS^1, N^{[2]}) \ar[u]&
}
\end{equation}
where~$f$ is induced by the squaring map,~$g$ is induced by the inclusion~$\rmS(N)\rightarrow N^{[2]}$ and all the other maps should be self-evident. It is clear that the homotopy type of this diagram is determined by the homotopy type of the diagram~$\rmS(N)\to N^{[2]}\to N\times N$, and therefore so is the homotopy limit of this diagram, which is~$\rmT_2\Emb(\rmS^1, N)$.
\end{proof}

\begin{corollary}\label{cor:main}
The homotopy~$(2n-7)$--type of the space~$\Emb(\rmS^1,N)$ does not depend on the smooth structure on the manifold~$N$.
\end{corollary}

\begin{proof}
The approximation map~$\Emb(M,N)\to\rmT_2\Emb(M,N)$ to the second layer in the Goodwillie--Weiss tower is~\hbox{$(2(n-m-2)+1-m)$}--connected. For the embeddings of the circle~$M=\rmS^1$, this shows that the approximation map~$\Emb(\rmS^1,N)\to\rmT_2\Emb(\rmS^1,N)$ is~$(2n-6)$--connected, so that both spaces share the same homotopy~$(2n-7)$--type. Therefore, the theorem implies the corollary.
\end{proof}

To end this section, we will describe the homotopy fiber of the map
\[
\rmT_2\Emb(\rmS^1, N)\to \rmT_1\Emb(\rmS^1, N)\simeq\Lambda\rmS(N)
\]
over a convenient basepoint, in low dimensions. Let us pick a point~$(x, \vec{u})\in\rmS(N)$, i.e., a point~$x\in N$ and a unit tangent vector~$\vec{u}$ at~$x$. Let the corresponding constant loop be our chosen basepoint of the space~$\Lambda \rmS(N)$. This point is not in the image of the homotopy equivalence~$\Imm(S^1, N)\to \Lambda\rmS(N)$, but it is connected by a path to the image of an unknot in a Euclidean neighborhood of~$x$. So we may think of our basepoint as representing a small unknot. 

Let~$\Omega N$ be the pointed loop space of~$N$ with our chosen basepoint~$x$. There is a map~$\Omega N\to N$, evaluating at the middle of a loop. Let~$(\Omega N)^{\tau}$ be the Thom space of the pullback of the tangent bundle of~$N$ along this evaluation map. The group~$\Sigma_2$ acts on~$(\Omega N)^{\tau}$ by reversing the direction of loops and by~$-1$ on the tangent bundle. Let~\hbox{$\rmQ Y=\colim_n\Omega^n\Sigma^n Y$} be the usual stable homotopy functor.

\begin{proposition}\label{prop: fiber}
Let~$N$ be a smooth manifold of dimension at least~$4$. The homotopy fiber of the forgetful map~$\rmT_2\Emb(\rmS^1, N)\to \Lambda \rmS(N)$ over a constant loop is related by a~$3$--connected map to the space 
\[
{\Map_*}_{\Sigma_2}\!\left(C, \Omega\rmQ(\Omega N)^{\tau}\right) 
\]
of equivariant pointed maps on the cofiber~\hbox{$C\cong\rmS^1\times\rmS^1/\rmS^1$} of the inclusion~\hbox{$\rmS(\rmS^1)\to(\rmS^1)^{[2]}$}.
\end{proposition}

\begin{proof}
By Theorem~\ref{thm:T2description} and/or diagram~\eqref{diagram: T2knots}, the homotopy fiber that we are interested in is equivalent to the total homotopy fiber of the following diagram
\[
\xymatrix{
\Map_{\Sigma_2}((\rmS^1)^{[2]},N^{[2]})  \ar[r]\ar[d] & 
 \Map_{\Sigma_2}((\rmS^1)^{[2]}, \ {N\times N})\ar[d]\\
 \Map(\rmS^1, N^{[2]}) \ar[r] & \Map(\rmS^1, \ {N\times N})   
}
\]
The calculation of the total fiber is pretty straightforward. If there is any subtlety, it has to do with basepoints and the (lack of) dependence on the choice of basepoint. We will calculate the total fiber by first taking fibers in the horizontal direction and then the vertical direction. 

Let~$F$ be the homotopy fiber of the map~$N^{[2]}\to N\times N$ over the basepoint~$(x, x)$. Because the inclusion~\hbox{$N\times N\setminus N\to N^{[2]}$} is a homotopy equivalence, the space~$F$ is equivalent to the homotopy fiber of the inclusion~\hbox{$N\times N\setminus N\to N\times N$}. There is a map
\[
\xymatrix@!0@C=60pt@R=30pt{
N\setminus\{x\}\ar[r]\ar[d]&N\times N\setminus N\ar[r]\ar[d]& N\ar@{=}[d]\\
N\ar[r]&N\times N\ar[r]&N
}
\]
of (horizontal) fibration sequences, and computing homotopy fibers vertically, we see that~$F$ is homotopy equivalent to the homotopy fiber of the inclusion~$N\setminus\{x\}\rightarrow N$. In particular, the space~$F$ is~$2$--connected. 

The total fiber that we are interested in is equivalent to the fiber of the following map
\[
\Map_{\Sigma_2}((\rmS^1)^{[2]}, F)\to  \Map(\rmS^1, F).
\]
Equivalently, we can write this map as follows, using~$\rmS(\rmS^1)\cong \Sigma_2\times \rmS^1$:
\begin{equation}\label{eq: fiber}
\Map_{\Sigma_2}((\rmS^1)^{[2]}, F)\to  \Map_{\Sigma_2}(\rmS(\rmS^1), F).
\end{equation}
Note that the basepoint of~$\Map_{\Sigma_2}(\rmS(\rmS^1), F)$ is not a constant map. (There is no constant~$\Sigma_2$--equivariant constant map into~$F$, since the action of~$\Sigma_2$ on~$F$ is free.) Rather, the basepoint is a map that is constant on each connected component of~$\rmS(\rmS^1)\cong \Sigma_2\times \rmS^1$. It sends one copy of~$\rmS^1$ to the point~$(x, \vec{u})$ and the other copy to~$(x, -\vec{u})$. Let us explain how we think of the two points~$(x, \pm{\vec{u}})$ as points in~$F$. Initially~$(x, \vec{u})$ was defined to be the chosen basepoint of~$\rmS(N)$. Since we have an inclusion~$\rmS(N)\hookrightarrow N^{[2]}$, the points~$(x, \pm\vec{u})$ can also be thought as points of~$N^{[2]}$. In fact, they are points in the fiber of the map~$N^{[2]}\to N\times N$ over the point~$(x,x)$. Therefore they naturally define points in the homotopy fiber of same map, which is~$F$. It follows that the homotopy fiber of the map~\eqref{eq: fiber} is the space of equivariant maps from~$(\rmS^1)^{[2]}$ to~$F$ that take one path component of the boundary of~$(\rmS^1)^{[2]}$ to~$(x, \vec{u})$ and the other path component to~$(x, -\vec{u})$. Notice that this map from the boundary of~$(\rmS^1)^{[2]}$ to~$F$ can be extended to a~$\Sigma_2$--equivariant map from all of~$(\rmS^1)^{[2]}$ to~$F$, because~$(\rmS^1)^{[2]}$ is two-dimensional and~$F$ is two-connected. In other words the fiber of~\eqref{eq: fiber} is not empty.

Next, we would like to stabilize. Let~$\widetilde\Sigma F$ be the {\it unreduced} suspension of~$F$. We use the unreduced suspension because~$F$ does not have a~$\Sigma_2$--equivariant basepoint. Let~$\widetilde \Omega \widetilde\Sigma F$ be the space of paths in~$\widetilde\Sigma F$ from the~``south pole'' to the~``north pole.'' (By our convention, the south pole is the basepoint of~$\widetilde\Sigma F$.) There is a natural map~$F\to \widetilde \Omega \widetilde\Sigma F$ that is~$5$--connected because the space~$F$ is~$2$--connected.~(This is a version of the Freudenthal suspension map for unpointed spaces.) It follows that the fiber of the map~\eqref{eq: fiber} is connected to the fiber of the map 
\begin{equation}\label{eq: weird}
\Map_{\Sigma_2}((\rmS^1)^{[2]}, \widetilde \Omega \widetilde\Sigma F)\to  \Map_{\Sigma_2}(\rmS(\rmS^1), \widetilde \Omega \widetilde\Sigma F)
\end{equation}
by a~$3$--connected map.

Next, we claim that one can replace~$\widetilde\Omega$ with the usual loop space~$\Omega$ in this map. To see this, observe that the homotopy fiber above can be identified with the space of~$\Sigma_2$--equivariant maps from~$(\rmS^1)^{[2]}\times I$ to~$\widetilde\Sigma F$ that agree with a prescribed map on the subspace
\[
\rmS(\rmS^1)\times I \underset{\rmS(\rmS^1)\times\partial I}{\cup}(\rmS^1)^{[2]}\times\partial I.
\]
The prescribed map is defined as follows. On~$\rmS(\rmS^1)\times I$  it is the composite~$\rmS(\rmS^1)\times I \to F\times I \to \widetilde \Sigma F$, where the first map is determined by the basepoint map~$\rmS(\rmS^1) \to F$ and the second map is the canonical quotient map. On the components of~$(\rmS^1)^{[2]}\times\partial I$ the prescribed map is constant with image the south pole and north pole, respectively. 
We claim that the prescribed map is~$\Sigma_2$--equivariantly homotopic to the constant map into the south pole: this claim follows from the fact that the basepoint map~$\rmS(\rmS^1)\to F$ can be extended to a~$\Sigma_2$--equivariant map from~$(\rmS^1)^{[2]}$ to~$F$, because the space~$F$ is~$2$--connected. It follows in turn that the homotopy fiber of the map~\eqref{eq: weird} is equivalent to the space of pointed~$\Sigma_2$--equivariant maps from~$(\rmS^1)^{[2]}\times I$ to~$\widetilde\Sigma F$ that agree with the{\it~trivial} map on the indicated subspace. This mapping space is
the homotopy fiber of the map
\[
\Map_{\Sigma_2}((\rmS^1)^{[2]},  \Omega \widetilde\Sigma F)\to  \Map_{\Sigma_2}(\rmS(\rmS^1),  \Omega \widetilde\Sigma F),
\]
which is obtained from~\eqref{eq: weird} by replacing the functor~$\widetilde \Omega$ with the ordinary loop space functor~$\Omega$. Note that now the basepoint of~$\Map_{\Sigma_2}(\rmS(\rmS^1),  \Omega \widetilde\Sigma F)$ is the constant map into the trivial loop.

Finally, we can stabilize. By the ordinary Freudenthal suspension theorem, the last homotopy fiber is mapped to the homotopy fiber 
of 
\begin{equation}
\Map_{\Sigma_2}((\rmS^1)^{[2]},  \Omega \rmQ  \widetilde\Sigma F)\to  \Map_{\Sigma_2}(\rmS(\rmS^1),  \Omega \rmQ  \widetilde\Sigma F).
\end{equation}
by a~$4$--connected map. It follows that the homotopy fiber that we are interested in is connected to the last homotopy fiber by a~$3$--connected map. Now the homotopy fiber is taken over the usual basepoint, given by the constant map. The homotopy fiber is equivalent to the space 
\[
{\Map_*}_{\Sigma_2}(C,  \Omega \rmQ  \widetilde\Sigma F).
\]
of equivariant pointed maps, where~$C$ the (homotopy) cofiber of the inclusion~\hbox{$\rmS(\rmS^1)\to(\rmS^1)^{[2]}$}.

It remains to identify the space~$\widetilde\Sigma F$ with the Thom space~$(\Omega N)^\tau$. Recall that~$F$ denotes the homotopy fiber of the canonical map~\hbox{$N^{[2]}\to N\times N$}. Consider the following cube, where~$\rmS(\Omega N)$ denotes the pullback of the spherical tangent bundle~$\rmS(N)$ along the evaluation map~\hbox{$\Omega N\to N$}. 
\[
\xymatrix@!0@C=50pt@R=30pt{
\rmS(\Omega N)\ar[dr] \ar[rr]\ar[dd] && \Omega N\ar[dr] \ar[dd]|\hole \\
&F \ar[dd]\ar[rr] && \ast\ar[dd]\\
\rmS(N)\ar[dr] \ar[rr]|\hole  && N \ar[dr] \\
&N^{[2]}\ar[rr]  &&  {N\times N}  
}
\]
By construction, all vertical squares are homotopy pullbacks. The bottom square is both a strict and a homotopy pushout diagram. 
By Mather's cube theorem~\cite[Thm.~25]{Mather}, the top square is a homotopy pushout square. Passing to its~(horizontal) homotopy cofibers, we get an equivalence~\hbox{$(\Omega N)^\tau\xrightarrow{\simeq}\widetilde\Sigma F~$}.
\end{proof}


\section{Applications to the fourth dimension}\label{sec:four}

Let~$N=N^4$ be a smooth connected~$4$--manifold. The tangent bundle of an oriented~$4$--manifold {\em is} determined by its topology~(see~\cite{Hirzebruch+Hopf} and~\cite{Dold+Whitney}): Any oriented~$4$--plane bundle over a~$4$--dimensional complex is determined by its second Stiefel--Whitney class~$w_2$, its first Pontryagin class~$p_1$, and its Euler class~$e$. For the tangent bundle, these classes are all topological invariants. This also implies Proposition~\ref{prop:sphere_bundles} in dimension~$4$.

In this section, we work out the implications of our general results for the spaces~$\Emb(\rmS^1,N^4)$:

\begin{corollary}\label{cor:main4}
The homotopy~$1$--type of the space~$\Emb(\rmS^1,N^4)$ does not depend on the smooth structure on the~$4$--manifold~$N^4$.
\end{corollary}

This allows us to compute the set of components (isotopy classes of embeddings), which was known before, and all the fundamental groups of the components, which is new, from the topology alone. We also get a lower bound on~$\pi_2$.

Let us start with~$\pi_0$. If~$M=\rmS^1$ and~$N^4$ is any~$4$--manifold, then the map~$\Emb(\rmS^1,N^4)\to\Imm(\rmS^1,N^4)$ is~$1$--connected, so that it induces a bijection between the sets of (path) components and an epimorphism on fundamental groups. For the set of components, we have Proposition~\ref{prop:immersion_homotopy_groups}, and get
\[
\pi_0\Emb(\rmS^1,N^4)\cong\pi_0\Imm(\rmS^1,N^4)\cong\pi_0\Lambda \rmS(N)\cong\pi_0\Lambda N^4,
\]
and that set is in natural bijection with the set of conjugacy classes of elements in the fundamental group~$\pi_1N^4$. Note that this set only depends on the homotopy type of the~$4$--manifold~$N^4$.

We can now turn to~$\pi_1$ and~$\pi_2$. For the moment, let us assume, for simplicity, that the manifold~$N^4$ is simply-connected. Then the space~$\Imm(\rmS^1,N^4)$ is connected and the fundamental group of that space is~$\rmH_2(N^4;\bbZ)$ by Proposition~\ref{prop:immersion_homotopy_groups}. Consequently, it is known that the space~$\Emb(\rmS^1,N^4)$ is path connected and that the fundamental group (and the first homology group) of that space surjects onto the abelian group~$\rmH_2(N^4;\bbZ)$. Corollary~\ref{cor:main4} lets us substantially improve on this estimate.

Our first application concerns homotopy~$4$--spheres, where the earlier estimate (``$\pi_1\Emb(\rmS^1,\Sigma^4)$ surjects onto the trivial group'') was contentless. Using Corollary~\ref{cor:main4} we now get:


\begin{proposition}\label{Prop:exotic-4-spheres}
If~$\Sigma^4$ is a homotopy~$4$--sphere, then the space~$\Emb(\rmS^1,\Sigma^4)$ of knots in~$\Sigma^4$ is simply-connected.
\end{proposition}

\begin{proof}
If~$\Sigma^4=\rmS^4$ happens to be the standard~$4$--sphere, then the statement is known to be true: the embedding space~$\Emb(\rmS^1,\rmS^4)$ is simply-connected~(see~\cite[Prop.~3.9]{Budney}, for instance): the embedding space~$\Emb(\rmS^1,\rmS^4)$ has the same homotopy~$1$--type as the space of linear embeddings, the Stiefel manifold~$\SO(5)/\SO(2)$. In general, Freedman has shown that every homotopy~$4$--sphere~$\Sigma^4$ is homeomorphic to the standard~$4$--sphere~$\rmS^4$, and then our Corollary~\ref{cor:main4} implies the result.
\end{proof}

\begin{remark}
It is known that the space~$\Emb(\rmS^1,\rmS^4)$ is not~$2$--connected. In fact, there is an isomorphism~$\pi_2\Emb(\rmS^1,\rmS^4)\cong\bbZ$, as shown by Budney~\cite[Prop.~3.9~(3)]{Budney} (mind the typo in the statement there). 
\end{remark}


As another application of our results, we will now show that there are~$4$--manifolds~$N$ for which the inclusion~\hbox{$\Emb(\rmS^1, N)\to \Imm(\rmS^1, N)$} has a (very) non-trivial kernel on~$\pi_1$ and~$\pi_2$:

\begin{example}
Let us consider the space~$\Emb(\rmS^1,\rmS^3\times\rmS^1)$. Because the target is parallelizable, we have an equivalence
\begin{equation}\label{eq:product}
\Imm(\rmS^1,\rmS^3\times\rmS^1)\simeq \Lambda (\rmS^3\times \rmS^3\times \rmS^1).
\end{equation}
It follows that the space~$\Imm(\rmS^1,\rmS^3\times\rmS^1)$, and therefore also the space~$\Emb(\rmS^1,\rmS^3\times\rmS^1)$ has a countably infinite number of connected components, indexed by the map induced between the fundamental~(or first homology) groups. Notice that the path components of~$\Imm(\rmS^1,\rmS^3\times\rmS^1)$ are homotopy equivalent to each other. We also have an analogous statement for the homotopy fibers of the map from~$\rmT_2 \Emb(\rmS^1,\rmS^3\times\rmS^1)$ to~$\rmT_1 \Emb(\rmS^1,\rmS^3\times\rmS^1)\simeq\Imm(\rmS^1,\rmS^3\times\rmS^1)$ over different path components:

\parbox{\linewidth}{\begin{proposition}\label{proposition: same fibers}
The homotopy type of the homotopy fiber of the map
\begin{equation}\label{eq: 2to1}
\rmT_2 \Emb(\rmS^1,\rmS^3\times\rmS^1)\longrightarrow\rmT_1 \Emb(\rmS^1,\rmS^3\times\rmS^1)
\end{equation}
is the same for every choice of basepoint.
In all cases, the homotopy fiber is equivalent to the space
\[
{\Map_{*}}_{\Sigma_2}(\rmS^1\times \rmS^1/\rmS^1, F)
\]
of equivariant pointed maps, where~$F$ is the homotopy fiber of the inclusion~$\rmS^3\times\rmS^1 \setminus{(1,1)}\to \rmS^3\times  \rmS^1$.
The action of the group~$\Sigma_2$ on~$F$ is defined via the action on~$\rmS^3$ and~$\rmS^1$ that sends an element of~$\rmS^1$ or~$\rmS^3$ to its inverse as a complex number or quaternion.
\end{proposition}
}

\begin{proof}
It is enough to verify the proposition for a choice of one basepoint in each path component.
We know that the following maps all induce a bijection on~$\pi_0$.
\[
\Emb(\rmS^1,\rmS^3\times\rmS^1)\longrightarrow\rmT_2 \Emb(\rmS^1,\rmS^3\times\rmS^1) 
\longrightarrow \rmT_1 \Emb(\rmS^1,\rmS^3\times\rmS^1) \longrightarrow \Lambda(\rmS^3\times\rmS^3 \times\rmS^1).
\]
We want to choose a convenient set of representative basepoints. For the purpose of this proof, let us identify~$\rmS^1$ with the circle of unit complex numbers and~$\rmS^3$ with the unit quaternions. The point~$(-1, -1)$ is our basepoint for~\hbox{$\rmS^3\times\rmS^1$}. Let~$i\colon \mathbb C\to \mathbb H$ be the inclusion of the complex numbers into the quaternions. Let~$\alpha_n\colon \rmS^1\to \rmS^3 \times \rmS^1$ be the map defined by~\hbox{$\alpha_n(z)=(i(z),z^n)$}. Then~$\{\alpha_n\mid n\in \mathbb Z\}$ is a complete set of representatives of the path components of~$\Emb(\rmS^1,\rmS^3\times\rmS^1)$ and therefore their images give a complete set of representatives of the path components in the other spaces, too.
We will show that the homotopy fibers of the map~\eqref{eq: 2to1} are pairwise homotopy equivalent for all these basepoints.

We have seen in Theorem~\ref{thm:T2description} and/or diagram~\eqref{diagram: T2knots}, that the homotopy fiber of~\eqref{eq: 2to1} is equivalent to the total homotopy fiber of the following square diagram.
\begin{equation}\label{eq: total fiber}
\xymatrix{
\Map_{\Sigma_2}((\rmS^1)^{[2]}, (\rmS^3\times \rmS^1)^{[2]})  \ar[r]\ar[d] & 
 \Map_{\Sigma_2}((\rmS^1)^{[2]}, \ {\rmS^3\times \rmS^1\times \rmS^3\times \rmS^1})\ar[d]\\
 \Map_{\Sigma_2}(\rmS(\rmS^1), (\rmS^3\times \rmS^1)^{[2]}) \ar[r] & \Map_{\Sigma_2}(\rmS(\rmS^1), \ {\rmS^3\times \rmS^1\times \rmS^3\times \rmS^1})   
}
\end{equation}
Of course ``the'' total homotopy fiber depends on a choice of basepoint in~$\Map_{\Sigma_2}((\rmS^1)^{[2]}, (\rmS^3\times \rmS^1)^{[2]})$. Our task is to compare the total homotopy fibers for basepoints given by the images of the embeddings~$\alpha_n$ defined above in~$\Map_{\Sigma_2}((\rmS^1)^{[2]}, (\rmS^3\times \rmS^1)^{[2]})$. The image of~$\alpha_n$ in~$\Map_{\Sigma_2}((\rmS^1)^{[2]}, (\rmS^3\times \rmS^1)^{[2]})$ is the map that sends a point~\hbox{$(z_1, z_2)\in \rmS^1\times \rmS^1\setminus \rmS^1$} to the point \[((i(z_1), z_1^n), (i(z_2), z_2^n))\in  (\rmS^3\times \rmS^1)\times  (\rmS^3\times \rmS^1)\setminus  \rmS^3\times \rmS^1.\]  

We will simplify diagram~\eqref{eq: total fiber} in three steps. 

(Step 1) Recall from Example~\ref{ex:circle} that there is a homeomorphism~$(\rmS^1)^{[2]}\cong \rmS^1\times [0,2\pi]$. The middle circle~$\rmS^1\times \{\pi\}$ corresponds under this homeomorphism to the subspace~$\widetilde{\rmS^1}:=\{(z, -z)\in (\rmS^1)^{[2]}\mid z\in\rmS^1\}$. With this homeomorphism it is clear that the circle~$\widetilde{\rmS^1}$ is a~$\Sigma_2$-equivariant strong deformation retract of~$(\rmS^1)^{[2]}$. From here it follows that the total homotopy fiber of~\eqref{eq: total fiber} is equivalent to the total homotopy fiber of the following diagram.
\begin{equation}\label{eq: middle}
\xymatrix{
\Map_{\Sigma_2}(\widetilde{\rmS^1}, (\rmS^3\times \rmS^1)^{[2]})  \ar[r]\ar[d] & 
 \Map_{\Sigma_2}(\widetilde{\rmS^1}, \ {\rmS^3\times \rmS^1\times \rmS^3\times \rmS^1})\ar[d]\\
 \Map_{\Sigma_2}(\rmS(\rmS^1), (\rmS^3\times \rmS^1)^{[2]}) \ar[r] & \Map_{\Sigma_2}(\rmS(\rmS^1), \ {\rmS^3\times \rmS^1\times \rmS^3\times \rmS^1})   
}
\end{equation}
To remember that the circle~$\widetilde{\rmS^1}$ with the antipodal action arose as the middle of~$(\rmS^1)^{[2]}$, we continue denoting the elements of~$\widetilde{\rmS^1}$ as pairs~$(z, -z)$, where~$z$ is a unit complex number. The vertical maps in the diagram are induced by the canonical~$\Sigma_2$-equivariant maps~$\rmS(\rmS^1)\cong \Sigma_2\times \rmS^1\to \widetilde{\rmS^1}$. Our task now is to compare the total fibers of the diagram~\eqref{eq: middle} with basepoints in~$\Map_{\Sigma_2}(\widetilde{\rmS^1}, (\rmS^3\times \rmS^1)^{[2]})~$ given by maps
\begin{equation}\label{eq:beta}
(z, -z) \mapsto ((i(z), z^n), (-i(z), (-z)^n))
\end{equation}
where~$n\in {\mathbb Z}$.

(Step 2) Notice that the boundary of~$(\rmS^3\times \rmS^1)^{[2]}$ is not in the image of any of the maps that serve as basepoints of the mapping space~$\Map_{\Sigma_2}(\widetilde{\rmS^1}, (\rmS^3\times \rmS^1)^{[2]})$. It follows that in diagram~\eqref{eq: middle} we may replace the space~$(\rmS^3\times \rmS^1)^{[2]}$ with the homotopy equivalent space~$(\rmS^3\times \rmS^1)^{2}\setminus \rmS^3\times\rmS^1$. So we are now interested in the total homotopy fiber of the following diagram.
\begin{equation}\label{eq: blowdown}
\xymatrix{
\Map_{\Sigma_2}(\widetilde{\rmS^1}, (\rmS^3\times \rmS^1)^{2}\setminus \rmS^3\times\rmS^1)  \ar[r]\ar[d] & 
 \Map_{\Sigma_2}(\widetilde{\rmS^1}, \ {(\rmS^3\times \rmS^1)^2})\ar[d]\\
 \Map_{\Sigma_2}(\rmS(\rmS^1), (\rmS^3\times \rmS^1)^{2}\setminus \rmS^3\times\rmS^1) \ar[r] & \Map_{\Sigma_2}(\rmS(\rmS^1), \ {(\rmS^3\times \rmS^1)^2})   
}
\end{equation}

(Step 3) Our last step is to simplify further the total homotopy fiber of~\eqref{eq: blowdown}, using the fact that~$\rmS^3\times \rmS^1$ is a Lie group. 

Suppose~$G$ is a Lie group. Let~$\widehat{G}$ denote the underlying space of~$G$, equipped with the~$\Sigma_2$--action that sends an element to its inverse. There is a fibration sequence
\[
G  \xrightarrow{g\mapsto  (g, g)}  G\times G  \xrightarrow{(g_1, g_2)  \mapsto  g_1^{-1}g_2 }  \widehat{G} 
\]
of spaces with~$\Sigma_2$--action, when we let $\Sigma_2$ act trivially on $G$ and by interchange-of-factors on $G\times G$. The pre-image of~$\widehat{G}\setminus \{e\}$ in~$G\times G$ is~$G\times G \setminus G$. It follows that there is a square
\[
\xymatrix@C=7em{
G\times G\setminus G\ar[d]_-{(g_1, g_2)\mapsto g_1^{-1}g_2}\ar[r]^-{\subset} & G\times G\ar[d]^-{(g_1, g_2)\mapsto g_1^{-1}g_2}\\
\widehat{G}\setminus\{e\}\ar[r]_-{\subset} & G
}
\]
that is both a pullback and a homotopy pullback.

Now taking~$G$ to be~$\rmS^3\times \rmS^1$, we conclude that the total fiber of~\eqref{eq: blowdown} is equivalent to the total fiber of the following square.
\begin{equation}\label{eq: cutdown}
\xymatrix{
\Map_{\Sigma_2}(\widetilde{\rmS^1}, \widehat\rmS^3\times \widehat\rmS^1 \setminus{(1,1)})  \ar[r]\ar[d] & 
 \Map_{\Sigma_2}(\widetilde{\rmS^1}, \ {\widehat\rmS^3\times \widehat\rmS^1})\ar[d]\\
 \Map_{\Sigma_2}(\rmS(\rmS^1), \widehat\rmS^3\times \widehat\rmS^1\setminus{(1,1)} \ar[r] & \Map_{\Sigma_2}(\rmS(\rmS^1), \ {\widehat\rmS^3\times \widehat\rmS^1})   
}
\end{equation}
Here~$\widehat \rmS^3$ and~$\widehat \rmS^1$ indicate the spheres~$\rmS^3$ and~$\rmS^1$ endowed with the action of the group~$\Sigma_2$ that sends an element~$z$ to its inverse~$z^{-1}$, and the map from~\eqref{eq: blowdown} to~\eqref{eq: cutdown} that induces an equivalence of total fibers is induced by the quotient map~$(\rmS^3 \times \rmS^1)\times(\rmS^3 \times \rmS^1)\to \widehat\rmS^3 \times \widehat\rmS^1$, defined by the formula~$(w_1, w_2)\mapsto w_1^{-1} w_2$. 

This finishes our simplification of diagram~\eqref{eq: total fiber}, and we can now describe the total homotopy fibers with respect to the various base points. Recall from~\eqref{eq:beta} that the representatives of the basepoints in the space~\hbox{$\Map_{\Sigma_2}(\widetilde{\rmS^1}, (\rmS^3\times \rmS^1)^2 \setminus \rmS^3\times \rmS^1)$} are given by the maps
\[
(z, -z)\mapsto ((i(z), z^n), (-i(z), (-z)^n))
\]
for~\hbox{$n\in\bbZ$}. The image of this basepoint in the upper left corner~$\Map_{\Sigma_2}(\widetilde{\rmS^1}, \widehat \rmS^3\times \widehat \rmS^1 \setminus{(1,1)})$ of the simplified square~\eqref{eq: cutdown} is the map that sends~$(z,-z)\in\widetilde \rmS^1$ to~$(-1,(-1)^n)$.
We see that, for any given $n\in\bbZ$, the induced basepoint is a {\em constant} map.
Therefore, for every $n\in\bbZ$, the total homotopy fiber of~\eqref{eq: cutdown} for the $n$--th basepoint is equivalent to the space 
\[
{\Map_{*}}_{\Sigma_2}(\rmS^1\times \rmS^1/\rmS^1, F_n)
\]
of pointed~$\Sigma_2$-equivariant maps, where~$\rmS^1\times \rmS^1/\rmS^1$ arises as the homotopy cofiber of the inclusion~\hbox{$\rmS(\rmS^1)\to\widetilde\rmS^1$}, and where the space~$F_n$ is the homotopy fiber of the inclusion~$\widehat \rmS^3\times \widehat \rmS^1 \setminus{(1,1)}\to \widehat \rmS^3\times \widehat \rmS^1$ over~$(-1, (-1)^n)$.
This already shows that there can be at most two different homotopy types of homotopy fibers: one for~$n$ even and one for~$n$ odd, because~$F_n$ depends only on the parity of~$n$ by definition. To resolve the remaining ambiguity, we remark that the spaces~$F_n$ are all~$\Sigma_2$--homotopy equivalent, because multiplication by~$-1$ induces a~$\Sigma_2$--equivariant homeomorphism from~$\widehat\rmS^1$ to itself that sends $+1$ to $-1$.
\end{proof}

It is elementary to deduce from the equivalence~\eqref{eq:product} that there are isomorphisms~$\pi_1(\Imm(\rmS^1,\rmS^3\times\rmS^1))\cong\bbZ$ as well as~\hbox{$\pi_2(\Imm(\rmS^1,\rmS^3\times\rmS^1))\cong\bbZ\oplus\bbZ$}, and all higher homotopy groups are finitely generated, too. In contrast:

\begin{proposition}\label{prop:homotopy_fiber}
Both~$\pi_1$ and~$\pi_2$ of each homotopy fiber of the map
\[
\Emb(\rmS^1,\rmS^3\times\rmS^1) \longrightarrow \Imm(\rmS^1,\rmS^3\times\rmS^1)
\]
are abelian and contain an infinitely generated free abelian group. In particular, they are not finitely generated.
\end{proposition}

\begin{proof}
Because of the equivalence~\hbox{$\rmT_1\Emb(\rmS^1,\rmS^3\times\rmS^1)\simeq\Imm(\rmS^1,\rmS^3\times\rmS^1)$} and the fact that the approximation~\hbox{$\Emb(\rmS^1,\rmS^3\times\rmS^1)\to\rmT_2 \Emb(\rmS^1,\rmS^3\times\rmS^1)$} is~$2$--connected, it is sufficient to prove the statement for the homotopy fibers of the map $\rmT_2 \Emb(\rmS^1,\rmS^3\times\rmS^1)\to\rmT_1 \Emb(\rmS^1,\rmS^3\times\rmS^1)$. By Proposition~\ref{proposition: same fibers} all of these homotopy fibers are equivalent to the mapping space~\hbox{${\Map_{*}}_{\Sigma_2}(\rmS^1\times \rmS^1/\rmS^1, F)$}, where~$F$ is the homotopy fiber of the map~$\rmS^3\times \rmS^1 \setminus (1,1) \to \rmS^3\times \rmS^1~$. 

As for the $\Sigma_2$--homotopy type of the homotopy fiber~$F$, the space~$\rmS^3\times \rmS^1 \setminus (1,1)$ is homotopy equivalent to~$\rmS^3\vee \rmS^1$, where we can take the wedge point to be~$(-1, -1)$. Let us recall that the group~$\Sigma_2$ is acting on~$\rmS^1$ and~$\rmS^3$ by taking each element to its inverse~(as elements in~$\bbC$ or~$\bbH$). 
The Whitehead product fibration $\Sigma(\Omega A\wedge \Omega B)\to A\vee B\to A\times B$ shows that we have equivalences
\begin{equation}\label{eq:F}
F\simeq
\Sigma(\Omega\rmS^3\wedge\Omega\rmS^1)\simeq
\bigvee_{n\in \mathbb Z} \Sigma \Omega \rmS^3\simeq
\bigvee_{n\in \mathbb Z} \rmS^3\vee \rmS^5\vee \rmS^7\ldots,
\end{equation}
where the last equivalence comes from the James splitting: the space~$\Sigma \Omega \rmS^3$ is homotopy equivalent to the wedge sum~$\rmS^3\vee\rmS^5\vee\rmS^7\ldots$. The action of the group~$\Sigma_2$ on the indexing set~$\mathbb Z$ sends~$n$ to its inverse~$-n$. 

As for the space~$\rmS^1\times\rmS^1/\rmS^1$, it fits into a cofibration sequence 
\[
{\Sigma_2}_+ \wedge\rmS^1 \overset{f}{\longrightarrow}
{\Sigma_2}_+ \wedge\rmS^1 \longrightarrow \rmS^1\times\rmS^1/\rmS^1,
\]
of pointed~$\Sigma_2$--equivariant maps. By adjunction, any such map~$f$ is described uniquely by an element in~\hbox{$\pi_1({\Sigma_2}_+ \wedge\rmS^1)$}, which is the free group on two generators, say~$a$ and~$\sigma(a)$, where~$\sigma$ is the non-trivial element of the group~$\Sigma_2$. The map~$f$ in question is~$a\cdot\sigma(a)$.

For any space~$F$, using the identifications~\hbox{${\Map_*}_{\Sigma_2}({\Sigma_2}_+ \wedge\rmS^1,F)\cong{\Map_*}(\rmS^1,F)=\Omega F$}, it follows that  with an action~$\sigma$ of the group~$\Sigma_2$, the mapping space~${\Map_*}_{\Sigma_2}\!\left(\rmS^1\times\rmS^1/\rmS^1, F\right)$ fits in a fibration sequence
\[
{\Map_*}_{\Sigma_2}\!\left(\rmS^1\times\rmS^1/\rmS^1, F\right) \longrightarrow \Omega F \xrightarrow{1+\sigma} \Omega F,
\]
where addition means loop multiplication.

Now let us take~$F$ to be what it was before, as in~\eqref{eq:F}.
We obtain a fibration sequence
\[
{\Map_*}_{\Sigma_2}\!\left(\rmS^1 \times \rmS^1/\rmS^1, F\right) \longrightarrow
\Omega (\bigvee_{n\in\mathbb Z}\rmS^3 \vee \rmS^5 \cdots) \longrightarrow 
\Omega(\bigvee_{n\in\mathbb Z} \rmS^3\vee\rmS^5 \cdots).
\]
Taking the homotopy long exact sequence, and focusing on~$\pi_2$ and~$\pi_1$, we obtain the following exact sequence.
\[
\pi_2{\Map_*}_{\Sigma_2}\!\left(\rmS^1 \times \rmS^1/\rmS^1, F\right) \longrightarrow 
\prod_{n=-\infty}^{\infty} \bbZ(n) \longrightarrow 
\prod_{n=-\infty}^{\infty} \bbZ(n) \longrightarrow 
\pi_1{\Map_*}_{\Sigma_2}\!\left(\rmS^1 \times \rmS^1/\rmS^1, F\right) \longrightarrow 0
\]
Here~$\bbZ(n)$ denotes a copy of the group~$\bbZ$ corresponding to the index~$n$. The homomorphism in the middle splits as a product of a homomorphism~$\bbZ(0)\to\bbZ(0)$, which we do not need to determine, 
and, for each~$n>0$, the homomorphism~$\bbZ(n)\times \bbZ(-n) \to \bbZ(n)\times \bbZ(-n)$ that sends a pair~$(i,j)$ to the pair~\hbox{$(i+j, i+j)$}. 

The group~$\pi_2{\Map_*}_{\Sigma_2}\!\left(\rmS^1 \times \rmS^1/\rmS^1, F\right)$ is abelian, because it is a~$\pi_2$, and by exactness of the sequence above, it surjects onto the kernel of the middle homomorphism, which obviously contains an infinitely generated free abelian group. Therefore, the group~$\pi_2{\Map_*}_{\Sigma_2}\!\left(\rmS^1 \times \rmS^1/\rmS^1, F\right)$ itself also contains an infinitely generated free abelian group.  As for the fundamental group~$\pi_1{\Map_*}_{\Sigma_2}\!\left(\rmS^1 \times \rmS^1/\rmS^1, F\right)$, we first note that the exact sequence implies that it is abelian as well, as the quotient of an abelian group. And the cokernel of the middle homomorphism, which is isomorphic to~$\pi_1{\Map_*}_{\Sigma_2}\!\left(\rmS^1 \times \rmS^1/\rmS^1, F\right)$, also contains an infinitely generated free abelian group. 
\end{proof}

\begin{corollary}\label{cor:kernels}
For~$j=1$ and~$j=2$, the kernel of the homomorphism
\[
\pi_j\Emb(\rmS^1,\rmS^3\times\rmS^1) \longrightarrow \pi_j\Imm(\rmS^1,\rmS^3\times\rmS^1)
\]
is abelian and contains an infinitely generated free abelian group. In particular, the kernels are not finitely generated.
\end{corollary}

We have also seen that the group~$\pi_1\Emb(\rmS^1,\rmS^3\times\rmS^1)$ contains an infinitely generated free abelian group. We refer to the Budney's and Gabai's more recent preprint~\cite{Budney+Gabai} for more information on these fundamental groups. Our methods allow us to obtain information on higher homotopy groups as well:

\begin{corollary}\label{cor:pi_two}
The group~$\pi_2\Emb(\rmS^1,\rmS^3\times\rmS^1)$ contains an infinitely generated free abelian group. In particular, it is not finitely generated.
\end{corollary}
\end{example}

It would be interesting to see a calculation showing an example of a {\it simply-connected}~$4$--manifold~$N$ for which the map~$\Emb(\rmS^1, N)\to \Imm(\rmS^1, N)$ has a non-trivial kernel on~$\pi_1$.~(Moriya's recent preprint~\cite{Moriya} contains restrictions that apply.) It is easy to show that the homotopy fiber of the map~$\rmT_2\Emb(\rmS^1, N)\to \Imm(\rmS^1, N)$, and therefore also of the inclusion~\hbox{$\Emb(\rmS^1, N)\to \Imm(\rmS^1, N)$}, has non-trivial~$\pi_1$ for many manifolds~$N$, including simply-connected ones. But we have not analyzed the long exact sequence in homotopy in enough detail to show that the map from the homotopy fiber to~$\rmT_2\Emb(\rmS^1, N)$ is non-zero on~$\pi_1$ for some simply-connected~$N$.


\section*{Acknowledgments}

The authors thank the Isaac Newton Institute for Mathematical Sciences, Cambridge, for support and hospitality during the programme `Homotopy harnessing higher structures' where work on this paper was undertaken. This work was supported by~EPSRC grant no~EP/K032208/1. The first author was supported by the Swedish Research Council, grant number 2016-05440.


\parskip11pt


\vfill

Gregory Arone\\
Department of Mathematics\\
Stockholm University\\
SE-106 91 Stockholm\\
SWEDEN\\
\href{mailto:gregory.arone@math.su.se}{gregory.arone@math.su.se}

Markus Szymik\\
Department of Mathematical Sciences\\
NTNU Norwegian University of Science and Technology\\
7491 Trondheim\\
NORWAY\\
\href{mailto:markus.szymik@ntnu.no}{markus.szymik@ntnu.no}

\end{document}